\documentclass[10pt, a4paper]{amsart}

\usepackage[unicode]{hyperref}
\usepackage{color}

\usepackage[english]{babel}
\usepackage{amssymb,amsmath,amsthm}
\usepackage{verbatim}

\usepackage[all]{xy}

\numberwithin{equation}{section}

\newtheorem{dummy}{dummy}[section]

\newtheorem{theorem}[dummy]{Theorem}

\newtheorem{lemma}[dummy]{Lemma}
\newtheorem{proposition}[dummy]{Proposition}
\newtheorem{remark}[dummy]{Remark}

\def\C{\mathbb C}
\def\P{\mathbb P}

\def\Q{\mathbb Q}

\def\Z{\mathbb Z}
\def\AA{\mathcal A}
\def\CC{\mathcal C}

\def\EE{\mathcal E}

\def\HH{\mathcal H}
\def\KK{\mathcal K}
\def\MM{\mathcal M}
\def\OO{\mathcal O}
\def\TT{\mathcal T}

\def\={\;=\;}
\def\bal{\begin{aligned}}
\def\eal{\end{aligned}}
\def\be{\begin{equation}\label}
\def\ee{\end{equation}}

\newcommand{\udot}{{\:\raisebox{3pt}{\text{\circle*{1.5}}}}}
\def \bullet {\udot}

\title{Exceptional collections of line bundles on the Beauville surface}

\author{Sergey Galkin,
Evgeny Shinder}

\thanks{Sergey Galkin: 
Laboratory of Algebraic Geometry, 
National Research University Higher School of Economics, 
7 Vavilova Str., Moscow, Russia, 117312;
Universit\"at Wien, Fakult\"at f\"ur Mathematik; 
Independent University of Moscow; Moscow Institute of Physics and Technology
{\tt Sergey.Galkin@phystech.edu} \\
Evgeny Shinder: \emph{Corresponding author}.
Max Planck Institute for Mathematics, Vivatsgasse 7,
Bonn 53177, Germany,
{\tt shinder@mpim-bonn.mpg.de}}

\begin{document}

\maketitle

\begin{abstract}
We construct quasi-phantom admissible subcategories in the derived category of 
coherent sheaves on the Beauville surface $S$. 
These quasi-phantoms subcategories appear as right orthogonals to subcategories generated by 
exceptional collections of maximal possible length 4 on $S$. 
We prove that there are exactly 6 exceptional collections consisting of line bundles (up to a twist) 
and these collections are spires of two helices.

\medskip

Keywords: exceptional collection, quasi-phantom category, Beauville surface

\end{abstract}

\section{Introduction}

Bounded derived categories of coherent sheaves on algebraic varieties, their
admissible subcategories and semiorthogonal decompositions have
been studied intensively by Bondal, Kapranov, Kuznetsov, Orlov, and others
\cite{Bon}, \cite{BK}, \cite{BO}, \cite{Kap}, \cite{Kuz06}, \cite{Kuz09}.

It has been questioned which additive invariants of admissible geometric triangulated categories 
are conservative, that is do not vanish for non-zero categories.
Non-vanishing of the Hochschild homology of geometric admissible categories
has been conjectured by Kuznetsov in \cite{Kuz09} and non-vanishing
of the Grothendieck group has been conjectured by Bondal in early 90's (unpublished).
On the contrary, existence of geometric categories with vanishing
Hochschild homology ({\it quasi-phantoms}) has been indicated by Katzarkov in
\cite{Kat} and existence of geometric categories with
vanishing Grothendieck group ({\it phantoms}) has been conjectured by Diemer, Katzarkov and Kerr \cite{DKK},
both motivated by considerations from mirror symmetry.

Let us consider the simplest interesting case, that of a complex smooth projective surface $S$ of general type.
On one hand such a surface is not expected to admit a full exceptional collection in its bounded derived 
category $D^b(S)$.
On the other hand exceptional collections of maximal possible length $\dim H^*(S,\Q)$ 
seem to exist at least in some cases when $p_g(S) = q(S) = 0$. 
In such a case the orthogonal complement to the category generated by
the exceptional collection has vanishing Hochschild homology \cite{Kuz09},
torsion Grothendieck group and generally rather mysterious structure.

The first counterexample to Kuznetsov's conjecture
was given by B\"ohning, Graf von Bothmer and Sosna, who constructed exceptional
collections of length $11$ on the classical Godeaux surface ($p_g = q = 0, K^2 = 1, b_2 = 9$) \cite{BBS}.
Alexeev and Orlov \cite{AO} came up with exceptional collections of length $6$ on Burniat surfaces ($p_g = q = 0, K^2 = 6, b_2 = 4$).
Some of the fake projective planes ($p_g = q = 0, K^2 = 9, b_2 = 1$) are expected to admit exceptional collections of length $3$ \cite{GKMS}[Section $3$].

In this paper we consider yet another surface with similar properties, the Beauville surface $S$ \cite{Bea}. 
$S$ is a surface of general type with $p_g = q = 0, K^2 = 8, b_2 = 2$, constructed as follows.
Let $C$ and $C'$ be two copies of the Fermat quintic
\[
X^5 + Y^5 + Z^5 = 0, 
\]
acted upon by $G = (\Z/5)^2$ in two different ways.
We consider the product surface $T = C \times C'$ with the diagonal $G$-action.
The latter action turns out to be free for an appropriate choice of $G$-actions on $C$ and $C'$.
The Beauville surface $S$ is defined as a quotient $T/G$.
According to \cite{BaC}, Theorem 3.7 there are two non-isomorphic surfaces that can be obtained
this way. We chose one of these two models which we describe in detail
in Section 1.

One can find useful the analogy between the Beauville surface $S$ and the
quadric surface, that is to think of Beauville surface as a sort of a fake quadric.
First of all these two surfaces have the same numerical invariants ($p_g = q = 0, K^2 = 8, b_2 = 2$).
Furthermore, we prove in Section $2.3$ that 
the Picard group of $S$ is generated modulo torsion
by the bundles $\OO(1,0)$, $\OO(0,1)$ which come as pull-backs
from the factors $C$ and $C'$.
The Riemann-Roch formula on $S$ implies that
\[
\chi( \OO(i,j) ) = ( i - 1 ) (j - 1) 
\]
also in analogy with the quadric on which we have minus signs replaced by the plus signs.
A line bundle $L \in Pic(S)$ is called \emph{acyclic} if $H^*(S,L) = 0$,
for example line bundles $\OO(-1,k)$ and $\OO(k,-1)$ are acyclic line bundles on a quadric $\P^1 \times \P^1$
for any $k$.
However unlike the quadric case 
there are only finitely many isomorphism classes of acyclic line bundles on $S$.

We list these line bundles in Section 3.2 (Lemma \ref{acyclicSets}) and use them to
construct six exceptional collections on $S$ 
of length $4$.
We prove that this list exhausts all the exceptional collections consisting
of $4$ line bundles up to a common twist by a line bundle (Theorem \ref{collections}). 
We compute dimensions of $Ext$-groups between elements of the collections in Proposition \ref{algebras}.
All of our exceptional collections in question are non-strict. Moreover in all of them
both $Ext^1$ and $Ext^2$ are present unlike the case of the Burniat surfaces
where only $Ext^2$ appears (\cite{AO}, Lemma 4.8). 
We also note that unlike the case of Burniat surfaces the exceptional collections
we present have no blocks, that is no groups of pairwise orthogonal elements.

Confirming the analogy between the Beauville surface and the quadric, it turns out
furthermore that line bundles in the exceptional collections on $S$ are all products of
powers of square roots $\KK(1,0)$, $\KK(0,1)$ of canonical classes coming from the factors
$C$ and $C'$.

We expect the existence of exceptional collections of line bundles
to hold for other product-quotient surfaces with $p_g=q=0, K^2 = 8$ (see e.g. \cite{BaP}) as well.
However we do not see at the moment whether there could be a uniform proof for that (see Remark \ref{canLattice}).
We plan to return to this question in the future.

\medskip

We thank Alexander Kuznetsov for reading a draft of the paper and providing us with
many useful comments and remarks.
We thank Ingrid Bauer, Arend Bayer, Ludmil Katzarkov, Mateusz Michalek, Dmitry Orlov, Yuri Prokhorov, Nicolo Sibilla,
Maxim Smirnov for helpful discussions.
We thank the referee for pointing out some typos and inaccuracies in the previous
version of the paper.

\medskip

The first author is partially supported by 
NSF Grant DMS0600800, 
NSF FRG Grant DMS-0652633, 
FWF Grant P20778, 
and an ERC Grant (GEMIS).
The second author is supported by the Max-Planck-Institut f\"ur Mathematik
and the SFB / Transregio 45 ``Periods, moduli spaces and arithmetic of algebraic varieties''
Bonn - Mainz - Essen.

\section{The Beauville surface and its properties}

\subsection{Generalities on $G$-equivariant line bundles}

We list general facts on $G$-linearized line bundles
and their cohomology (see \cite{Mum} for details).

Let $G$ be a finite group acting on a smooth projective variety
$X/\C$.
The equivariant Picard group $Pic^G(X)$ is the group of isomorphism classes of 
$G$-linearized line bundles on $X$.
The equivariant Picard group is related to the ordinary Picard
group via an exact sequence
\be{picG}
0 \to \widehat{G} \to Pic^G(X) \to Pic(X)^G,
\ee
where $\widehat{G} = Hom(G, \C^*)$ is the group of characters 
and the first arrow associates to a character $\chi: G \to \C^*$
a trivial line bundle with the $G$-action induced by $\chi$.

\medskip

Suppose $G$ is abelian; then we can describe the equivariant Picard group in terms
of $G$-invariant divisors on $X$. 

\begin{lemma}
Let $G$ be a finite abelian group. Then the image of $Pic^G(X)$ in $Pic(X)^G$ consists of
equivalence classes of $G$-invariant divisors and (\ref{picG}) rewrites as
\be{picGab}
0 \to \widehat{G} \to Pic^G(X) \to \frac{Div(X)^G}{rational \, equivalence} \to 0.
\ee
\end{lemma}
\begin{proof}
We need to prove that for a $G$-linearized line bundle $L$ there exists a section $s$
with a $G$-invariant divisor $div(s)$. 
Let $W$ be an arbitrary finite-dimensional invariant subspace of meromorphic
sections of $L$.

Since $G$ is abelian, we may assume $W$ is one-dimensional, $W = \C \cdot s$.
Now $s$ is a $G$-eigensection, which is equivalent to $div(s)$ being $G$-invariant.
\end{proof}

\medskip

If $G$ is a finite group (not necessarily abelian) acting freely on $X$, then we have
an etale covering of smooth projective varieties
\[
\pi: X \to X/G.
\]
In this case specifying
a line bundle $L$ on $X/G$ is the same as
specifying a  line bundle $\widetilde{L}=\pi^* L$ on $X$ together
with additional structure of $G$-linearization. This way we get an identification
\[
Pic^G(X) = Pic(X/G).
\]

For any line bundle $L$ on $X/G$
the groups $H^i(X,\pi^* L)$ have a natural structure of $G$-representations
and we have canonical isomorphisms
\[
H^i(X/G,L) = H^i(X,\pi^* L)^G.
\]

\medskip

For our computations 
we need an equivariant version of the Serre duality.
For any $G$-linearized line bundle on $X$ we have an isomorphism of $G$-representations:
\be{Serre}
H^k( X, L ) \cong (H^{dim(X)-k}( X, L^* \otimes \omega_X ))^*.
\ee

\begin{lemma}\label{canClass}
Let $V$ be an $n+1$-dimensional representation of a finite group $G$.
Then we have an isomorphism of $G$-linearized line bundles on $\P(V)$:
\[
\omega_{\P(V)} \cong \OO(-n-1)( \det V^* ).
\]
\end{lemma}
\begin{proof}
The claim follows by taking the determinant of the Euler exact sequence
of $G$-linearized line bundles on $\P(V)$
\[
0 \to \Omega^1_{\P(V)} \to \OO(-1) \otimes V^* \to \OO \to 0.
\]
\end{proof}

In the notation of Lemma \ref{canClass} 
let $F$ be an invariant section of $\OO(d)$ on $\P(V)$ and $X$ be the hypersurface $F = 0$.
Then there is a standard adjunction formula giving an isomorphism of $G$-linearized line bundles on $X$:
\be{canClassCor}
\omega_{X} \cong \OO(d-n-1)( \det V^* ).
\ee

\subsection{Equivariant Fermat quintics}

In what follows $G$ is an abelian group 
\[
G = (\Z/5)^2 = \Z/5 \cdot e_1 \oplus \Z/5 \cdot e_2
\]
acting on a three dimensional vector space $V$ with
induced action on $\P^2 = \P(V)$ given by
\[\bal
e_1 \cdot (X:Y:Z) & \= (\zeta_5 X: Y:Z)  \\
e_2 \cdot (X:Y:Z) & \= (X: \zeta_5 Y:Z), 
\eal\]
where $\zeta_5$ is the $5$-th root of unity.
Let $C$ be the plane $G$-invariant Fermat quintic curve
\[
X^5 + Y^5 + Z^5 = 0.
\]

\medskip

We consider the scheme-theoretic quotient $C/G$ which is isomorphic
to $\P^1$ and the quotient map
\[
\pi: C \to \P^1
\]
of degree $25$. Explicitly we may pick coordinates on $\P^1$ such that $\pi$ is
given by the formula
\[
\pi(X:Y:Z) = (X^5:Y^5). 
\]

One easily checks that there are three ramification points on $\P^1$ corresponding
to the orbits where $G$ acts non-freely:

\be{Di}\bal
D_1 & \= \{ (0: -\zeta_5^j : 1), \, j = 0 \dots 4\} \\ 
D_2 & \= \{ (-\zeta_5^j : 0: 1), \, j = 0 \dots 4\} \\ 
D_3 & \= \{ (\zeta_5^j: -\zeta_5^j : 0), \, j = 0 \dots 4\} \\ 
\eal\ee

Stabilizers of the points in $D_i, \; i=1,2,3$ are equal to 
\be{StabC}\bal
G_1 &\= \Z/5 \cdot e_1  \\
G_2 &\= \Z/5 \cdot e_2 \\
G_3 &\=\Z/5 \cdot (e_1 + e_2)
\eal\ee
respectively.

\begin{lemma}\label{picGC}
The equivariant Picard group $Pic^G(C)$ splits as a direct sum
\[
Pic^G(C) = \widehat{G} \oplus \Z \cdot \OO(1).
\]
\end{lemma}

\begin{proof}
The claim follows from the exact sequence (\ref{picGab}).
Indeed any $G$-invariant divisor is a combination of $G$-orbits on $C$. 
Any orbit is either a smooth fiber of $\pi$ consisting of $25$ points or
one of the divisors (\ref{Di}) consisting of $5$ points.
Since $D_1$, $D_2$, $D_3$ are hyperplane sections of $C$
they give rise to the same element $\OO(1)$ in the Picard group $Pic(C)$.
All the generic fibers are of $\pi$ are linearly equivalent to each other,
and also equivalent to $\OO(5)$.

Therefore the third term
in the exact sequence (\ref{picGab}) is $\Z \cdot \OO(1)$ and (\ref{picGab}) splits
giving the required decomposition.
\end{proof}

We introduce some notation which will help us
to keep track of characters appearing in the cohomology representations.
Note that the Grothendieck ring of the category of
$\Z_+$-graded representations
of $G$ is isomorphic to $\Z[q,x,y] / (x^5-1,y^5-1)$.
Thus to any $\Z_+$-graded $G$-representation $W$ we can attach a polynomial
\be{brk}
[W] \in K_0(Rep_{\Z_+}(\Z/5)^2) = \Z[q,x,y] / (x^5-1,y^5-1).
\ee

By definition we have the following properties of the polynomial $[W]$:
\[\bal\;
[W \oplus W'] & \=  [W] + [W'] \\
[W \otimes W'] & \=  [W] \cdot [W'] \\
[W^*] & \=  [W] \Big|_{x=x^4, y=y^4}. \\
\eal\]

\medskip

Later we will use the same bracket notation $[i,j]$, $i,j \in \Z/5$ for the character $e_1 \mapsto \zeta_5^i, e_2 \mapsto \zeta_5^j$
which will hopefully not lead to a confusion. For example we have
\[
[W[i,j]] = [W] \cdot x^i y^j. 
\]

\medskip

We now proceed to computing cohomology groups of line bundles $\OO(n), n \le 5$ on $C$ taking into account
the $G$-action. 
For $n \le 4$ we have
\[
H^0(C, \OO(n)) \cong H^0(\P^2, \OO(n)) = \bigoplus_{i,j \ge 0, \, i + j \le n} \C \cdot X^i Y^j Z^{n-i-j}.
\]
For $n = 5$ we quotient out the representation space $H^0(\P^2,\OO(5))$ by the relation $X^5 + Y^5 + Z^5 = 0$.
Thus we have
\be{cohC0}\bal\;
[H^0(C,\OO(n))] &\= \sum_{i,j \ge 0, \, i + j \le n} x^i \, y^j, \; 0 \le n \le 4 \\
[H^0(C,\OO(5))] &\= \sum_{i,j \ge 0, \, i + j \le 5} x^i \, y^j - 1. 
\eal\ee

In order to compute $H^1(C, \OO(n))$ we first use the adjunction (\ref{canClassCor}):
\[\bal
V^* &\= \Gamma(\P(V),\OO(1)) = \C \cdot X \oplus \C \cdot Y \oplus \C \cdot Z \cong [1,0] \oplus [0,1] \oplus [0,0]  \\
\det(V^*) &\= [1,0] \otimes [0,1] \otimes [0,0] = [1,1] \\
\omega_C &= \OO(2)[1,1], \\
\eal\]
so that by Serre duality (\ref{Serre}) we have
\[
H^1(C, \OO(n)) \cong H^0(C, \OO(2-n)[1,1])^* = H^0(C, \OO(2-n))^*[4,4],    
\]
which in terms of polynomials implies that
\[
[H^1(C,\OO(n))](x,y) = [H^0(C,\OO(2-n))](x^4,y^4) \cdot x^4 \, y^4.
\]
A short computation shows that
\be{cohC1}\bal\;
[H^1(C,\OO)] & \= q(x^4 y^4 + x^4 y^3 + x^3 y^4 + x^4 y^2 + x^3 y^3 + x^2 y^4) \\
[H^1(C,\OO(1))] & \= q(x^4 y^4 + x^4 y^3 + x^3 y^4) \\
[H^1(C,\OO(2))] & \= q x^4 y^4 \\
[H^1(C,\OO(n))] & \= 0, \; n \ge 3. \\
\eal\ee

We introduce the curve $C'$ which is defined by the same equation
\[
X^5 + Y^5 + Z^5 = 0 
\]
as $C$ but has a different $G$-action. We pick the $G$-action on $C'$ to be defined as 
\[\bal
e_1 \cdot (X:Y:Z) &\= (\zeta_5^2 X: \zeta_5^4 Y : Z) \\
e_2 \cdot (X:Y:Z) &\= (\zeta_5 X: \zeta_5^3 Y : Z) 
\eal\]
For this action points in divisors $D_i, \; i=1,2,3$ defined as in (\ref{Di})
have stabilizers
\be{StabCp}\bal
G_1' &\= \Z/5 \cdot (e_1 + 2 e_2) \\
G_2' &\= \Z/5 \cdot (e_1 + 3 e_2) \\
G_3' &\=\Z/5 \cdot (e_1 + 4 e_2)
\eal\ee
respectively.


It follows from the construction that for any $n \in \Z$ we have a formula
\be{cohCp}
[H^*(C',\OO(n))](q,x,y) = [H^*(C,\OO(n))](q,x^2 y,x^4 y^3)
\ee
and that the canonical class on $C'$ is equal to $\OO(2)[1,4]$.

We introduce the notation
\[\bal
\KK_C(1) &= \OO_C(1)[3,3] \\
\KK_{C'}(1) &= \OO_{C'}(1)[3,2] \\
\eal\]
for the unique square roots of the canonical classes on $C$ and $C'$ respectively.

\subsection{Line bundles and cohomological invariants of the Beauville surface}

We let $T = C \times C'$ with the diagonal $G$-action. Since the stabilizers in
(\ref{StabC}) and (\ref{StabCp}) are distinct, the $G$-action on $T$ is free.
One can check that the corresponding smooth quotient Beauville surface $S = T/G$ is of general type with $p_g = q = 0, K^2 = 8$
(Chapter X, Exercise 4 in \cite{Bea}).
The Noether formula gives $b_2 = 2$.
Since $p_g = q = 0$, the exponential exact sequence gives an identification
\[
Pic(S) = H^2( S, \Z ).
\]
Modulo torsion $Pic(S)$ is an indefinite unimodular lattice of rank $2$, that is a hyperbolic plane.

We introduce $G$-linearized line bundles $\OO(i,j)$ and $\KK(i,j)$ for $i,j \in \Z$
as follows:
\[\bal
\OO(i,j) &= p_1^*(\OO(i)) \otimes p_2^*(\OO(j)) \\
\KK(i,j) &= p_1^*(\KK(i)) \otimes p_2^*(\KK(j)) = \OO(i,j)[3i+3j,3i+2j].\\
\eal\]

We will often prefer to work with the lattice $\KK(i,j)$ since the exceptional
collections we write down in Section 3 are all contained in this lattice.

We note however that $\KK(i,j)$ and $\OO(i,j)$ differ by a torsion line bundle hence are 
equivalent from the point of view of intersection pairing.
In particular in the following Proposition $\OO(i,j)$ can be replaced by $\KK(i,j)$
(with an obvious exception of the second claim).

\begin{proposition}\label{lineBundles}
1. The Picard group of $S$ splits as
\[
Pic(S)( = Pic^G(T)) = \widehat{G} \cdot [\OO] \oplus \Z \cdot [\OO(1,0)] \oplus \Z \cdot [\OO(0,1)].
\]

2. The canonical class $\omega_S$ is equal to $\KK(2,2) = \OO(2,2)[2,0]$.

3. The intersection pairing is given by
\[
(\OO(i_1,j_1)(\chi_1) \cdot \OO(i_2,j_2)(\chi_2)) = i_1 j_2 + j_1 i_2.
\]

4. The Euler characteristic of a line bundle $L = \OO(i,j)(\chi)$ is equal to $(i-1)(j-1)$.

\end{proposition}
\begin{proof}
Let us first prove that
\be{inters1}
(\OO(1,0) \cdot \OO(0,1)) = 1.
\ee
For that we pull-back the intersection to $T$:
\[\bal
25 \cdot (\OO(1,0) \cdot \OO(0,1))_S &\= (\pi^* \OO(1,0) \cdot \pi^* \OO(0,1))_T \\ 
&\= (5 [pt \times C'] \cdot 5 [C \times pt])_T = 25,
\eal\]
which implies (\ref{inters1}). Since we also obviously have
\be{inters2}
(\OO(1,0)^2) = (\OO(0,1)^2) = 0, 
\ee
it follows that $\OO(1,0)$ and $\OO(0,1)$ span a hyperbolic plane
and therefore generate the whole Picard group modulo torsion. 

To prove the first claim we use the fact that $H_1(S) = (\Z/5)^2$ \cite{BaC}, Theorem 4.3, (4),
which implies that
\[
Pic(S)_{tors} = H^2(S,\Z)_{tors} = H_1(S,\Z)_{tors} = (\Z/5)^2.
\]
Since by (\ref{picGab}) $\widehat{G} \cong (\Z/5)^2$ is contained in $Pic(S)$, $Pic(S)_{tors} \cong \widehat{G}$ 
and we get a decomposition
\[
Pic(S) = \widehat{G} \cdot [\OO] \oplus Pic(S)/tors = \widehat{G} \cdot [\OO] \oplus \Z \cdot [\OO(1,0)] \oplus \Z \cdot [\OO(0,1)].
\]

The second claim follows from
\[
\omega_S = p_1^* \omega_C \otimes p_2^* \omega_{C'} = \KK(2,0) \otimes \KK(0,2) = \OO(2,0)[1,1] \otimes \OO(0,2)[1,4].
\]

The third claim of the Lemma follows from (\ref{inters1}),(\ref{inters2}), and the fact that twisting by torsion
classes does not affect the intersection form.

To check the fourth claim we use Riemann-Roch formula:
\[\bal
\chi(L) &\= 1 + \frac{(L \cdot L \otimes \omega_S^*)}{2} \\
& \=  1 + \frac{(\OO(i,j)(\chi) \cdot \OO(i-2,j-2)(\chi- [2,0]))}{2} \\
& \=  1 + \frac{(i(j-2)+j(i-2)}{2} \\
& \= (i-1)(j-1).
\eal\]

\end{proof}

We have a K\"unneth-type formula for isomorphism classes of graded representations
(recall the notation from (\ref{brk})): 
\be{Kunneth}
[H^*(T,\KK(i,j))](q,x,y) = [H^*(C,\KK(i))](q,x,y) \cdot [H^*(C',\KK(j))](q,x,y),
\ee
and the analogous formula with $\KK(i,j)$ replaced by $\OO(i,j)$.
This is simply a reformulation of the K\"unneth formula
\[
H^*(C \times C', p_1^* L_1 \otimes p_2^* L_2) = H^*(C, L_1) \otimes H^*(C', L_2).
\]
with the $G$-action on both sides taken into account.

\medskip

In the following Lemma we perform necessary computations which will be used later for computing Hochschild homology of $S$
as well as cohomology of $dg$-algebras of the exceptional collections on $S$.

\begin{lemma}\label{cohKK}
Some cohomology ranks $h^0(\KK(i,j)) + q h^1(\KK(i,j)) + q^2 h^2(\KK(i,j))$ are given in the table:
\[
\begin{tabular}{|c|c|c|c|c||c||c|c|c|c|c|c|}
\hline
$_j \;^i$ & -3 	&-2 &  -1 	& 0 &  1    &   2  & 3 & 4 & 5 \\
\hline
4 	&   &  & & & $0 $	    & $3$ &   $6$ &  $9$ & \\
\hline
3 	&   &  & & & $3+3q$   & $3+q$ & $4$ & $6$ & $8$ \\
\hline
2 	&   &  & & & $0$ & $q^2$ & $3+q$ & $3$ & \\
\hline
\hline
1 & & $3q^2+3q$ & $0$ & $0$ & $0$ & $0$ & $0$ & $3+3q$ &\\
\hline
\hline
0 & & $3q^2$ & $3q^2 + q$ & $1$ & $0$ &  & & & \\
\hline
-1 & $8q^2$ &  $6q^2$ & $4q^2$ & $3q^2 + q$ & $3q^2 + 3q$ &  &  & & \\
\hline
-2 & & $9 q^2$  & $6q^2$ & $3q^2$ & $0$ &  &  & & \\
\hline
\end{tabular}
\]
\end{lemma}
\begin{proof}
The entries of the table are in agreement with the Serre isomorphism
\[
h^p(S, \KK(i,j)) = h^{2-p}(S, \KK(2-i,2-j)),
\]
therefore it is sufficient to consider $i,j$ from the
table with $i,j \ge 1$.
The Euler characteristic of $\KK(i,j)$ is equal to $(i-1)(j-1)$.
By Kodaira vanishing theorem there is no higher cohomology for $i, j \ge 3$.
The rest is done using 
the K\"unneth formula (\ref{Kunneth})
and (\ref{cohC0}), (\ref{cohC1}), (\ref{cohCp}) which we use
to compute:
\be{cohPolyn}\bal\;
[H^*(C, \KK(1))] &\= x^4y^3 + x^3y^4 + x^3y^3 + q x^2y^2 + q x^2y + q xy^2 \\
[H^*(C', \KK(1))] &\= x^3y^2 + y^3 + x^2  + q x^3 + q y^2 + q x^2y^3 \\
[H^*(C, \KK(2))] &\= x^3y + x^2y^2 + xy^3 + x^2y + xy^2 + xy + q \\
[H^*(C', \KK(2))] &\= x^2y^3 + xy^4 + x^4 + x^3 + y^2 + y + q\\
[H^*(C, \KK(3))] &\= x^4y^4 + x^4y^2 + x^2y^4 + x^4y + xy^4 + x^4 + y^4 + x + y + 1 \\
[H^*(C', \KK(3))] &\= x^4y^3 + x^3y^4 + x^3y^3 + x^4y + x^2y^2 + y^4 + x^2y + xy^2 + x + 1 \\
[H^*(C, \KK(4))] &\= x^4y^4 + x^4y^3 + x^3y^4 + x^4y^2 + x^3y^3 + x^2y^4 + x^3y^2 + x^2y^3 + \\ 
& \; \; \; \; + x^2y^2 + x^3 + x^2y + xy^2 + y^3 + x^2 + y^2 \\
[H^*(C', \KK(4))] &\= x^4y^4 + x^4y^2 + x^2y^4 + x^4y + x^3y^2 + x^2y^3 + x^3y + xy^3 + \\
& \; \; \; \; + y^4 + x^3 + y^3 + x^2 + xy + y^2 + x. \\
\eal\ee

\end{proof}

\begin{lemma}
The Hochschild cohomology groups $HH^*(S,\C) = \oplus_{p+q=*} H^p(S, \Lambda^q \TT_S)$ of $S$ are given below.
\[\bal
HH^0(S) &\= \C \\
HH^1(S) &\= 0 \\
HH^2(S) &\= 0 \\
HH^3(S) &\= H^2(S, \TT_S) = \C^6 \\
HH^4(S) &\= H^2(S, \Lambda^2 \TT_S) = \C^9. \\
\eal\]
\end{lemma}
\begin{proof}
We have
\[
H^p(S, \Lambda^q \TT_S) =  H^p(T, \Lambda^q \TT_T)^G
\]
and
\[\bal
\TT_T &= p_1^*\TT_C \oplus p_2^*\TT_{C'} = \KK(-2,0) \oplus \KK(0,-2) \\
\Lambda^2 \TT_T &= p_1^*\TT_C \otimes p_2^*\TT_{C'} = \KK(-2,-2). \\
\eal\]

Now the cohomology groups in question are found in the table of Lemma \ref{cohKK}.

\end{proof}

Next we would like to compute the Grothendieck group $K_0(S)$ of the Beauville surface $S$.
By the results of Kimura \cite{Kim}, Bloch conjecture is known for all surfaces with $p_g = 0$ which 
admit a covering by a product of curves, hence $CH_0(S) = \Z$ for the Beauville surface $S$. 
Therefore by Lemma \ref{K0lemma} below the Grothendieck group of $S$ has a decomposition
\[
K_0(S) = \Z^4 \oplus (\Z/5)^2.
\]

\begin{lemma}\label{K0lemma}
Let $X$ be a smooth projective surface such that the degree morphism
$CH_0(X) \to \Z$
is an isomorphism. Then we have a (non-canonical) isomorphism
\[
K_0(X) \cong \Z^2 \oplus Pic(X).
\]
\end{lemma}
\begin{proof}
Consider the topological filtration $F^i \subset K_0(X)$ given by the codimension of support \cite{Ful}. By Riemann-Roch theorem without denominators \cite{Ful}
we have \[\bal
F^0 / F^1 &\cong \Z \\
F^1 / F^2 &\cong Pic(X) \\
F^2 &\cong CH_0(X) \cong \Z. \\
\eal\]

Extension $0 \to F^1 \to F^0 \to \Z \to 0$ always splits for group-theoretic reasons,
so 
\[
K_0(S) = \Z \oplus F^1.
\]

We have a short exact sequence
\be{seqF1}
0 \to \Z \overset{i}\to F^1 \to Pic(X) \to 0.
\ee
We have to prove that (\ref{seqF1}) splits, that is
there exists a retraction $F^1 \to \Z$.
In general such a retraction exists whenever 
the image of $i(1)$ in $F^1/tors$ is not divisible by any integer $a>1$.
Recall that $i(1) = [\OO_P] \in F^1 \subset K_0(S)$ where $P$ is a point of S.
Assume that $[\OO_P] = a \cdot A + \alpha$, where $\alpha$ is a torsion element.
Then
\[
1 = \chi(\OO, \OO_P) = \chi(\OO, aA + \alpha) = a \chi(\OO,A)
\]
since $a$ is positive integer and $\chi(\OO,A)$ is integer last equality implies $a = 1$.

\end{proof}

\section{Exceptional collections on the Beauville surface}

\subsection{Numerically exceptional collections and helices}

We call a sequence of line bundles
\[
L_1, \dots, L_n
\]
on a variety \emph{numerically exceptional} if for all $j > i$
\[
\chi(L_j, L_i) = \sum_l (-1)^l \dim Ext^l( L_j, L_i ) = 0.
\]

Any exceptional collection is obviously numerically exceptional as well.
We note that in order to speak about numerically exceptional collections
we only need to consider classes of $L_i$'s modulo torsion.
This implies that a sequence
$L_1, \dots, L_n$ forms a numerically exceptional collection on $S$
if and only if any twist $L_1(\chi_1), \dots, L_n(\chi_n)$ does. 
In particular we will not make a distinction between $\OO(i,j)$ and $\KK(i,j)$
when investigating numerically exceptional collections.

\begin{lemma}\label{numExceptional} A sequence
\[
\OO, L_1, L_2, L_3 
\]
of line bundles on $S$ is numerically exceptional if and only if
it belongs to one of the following four numerical types:
\[\bal
(I_c) & \; \OO, \OO(-1,0), \OO(c-1,-1), \OO(c-2,-1), \; c \in \Z \\
(II_c) & \; \OO, \OO(0,-1), \OO(-1,c-1), \OO(-1,c-2), \; c \in \Z \\
(III_c) & \; \OO, \OO(-1,c), \OO(-1,c-1), \OO(-2,-1), \; c \in \Z \\
(IV_c) & \; \OO, \OO(c,-1), \OO(c-1,-1), \OO(-1,-2), \; c \in \Z. \\
\eal\]
\end{lemma}

We note that $I_0 = III_0$, $II_0 = IV_0$ and also that types $I_c$ and $II_c$
are characterized by the property $L_3 \cong L_1 \otimes L_2$.

\begin{proof}
By Proposition \ref{lineBundles}(4) the sequence
\[
\OO, \OO(a_1,b_1), \OO(a_2,b_2), \OO(a_3,b_3)
\]
is numerically exceptional if and only if all of the vectors $(a_i,b_i)$,
$(a_j - a_i, b_j - b_i), \, j > i$ have one of the coordinates equal to -1.
The rest of the proof is left to the reader.
\end{proof}

If we consider a general sequence of line bundles
\be{seq1}
L_0, L_1, L_2, L_3
\ee
on $S$, then it is (numerically) exceptional if and only if
\be{seq2}
\OO, L_1 \otimes L_0^*, L_2 \otimes L_0^*, L_3 \otimes L_0^* 
\ee
is (numerically) exceptional.
We say that the sequence (\ref{seq1}) is of type $I_c$, $II_c$, $III_c$ or $IV_c$
if (\ref{seq2}) is of this type.

\medskip

In order to study exceptional collections on $S$ more efficiently we will
use so-called helices (\cite{GR}, \cite{Bon}, \cite{BP}).
We call a sequence $E_\bullet = (E_i, i \in \Z)$ of sheaves on a smooth variety $X$ 
\emph{a helix of period $n$} if
\[
E_{i-kn} = E_i \otimes \omega_X^{\otimes k} 
\]
for all $0 \le i \le n-1, \; k \in \Z$.
\footnote{The definition of helix we use coincides with that from \cite{Bon} up to shifts
which we have dropped for convenience.
The definition of helix in \cite{BP} which is given in terms of mutations rather 
than the Serre functor differs from ours since the collections we consider are not full.}
Given a sequence $E_0,\dots,E_{n-1}$ of sheaves
on $X$ we can extend it to a helix by the formula above.

Any subsequence of a helix consisting of $n$ consecutive elements $E_a, E_{a+1}, \dots,
E_{a+n-1}$ will be called \emph{a spire}.
By Serre duality an arbitrary spire of a helix is a (numerically) exceptional collection if
and only if $E_0,\dots,E_{n-1}$ is a (numerically) exceptional collection.
We will sometimes represent a helix $E_\bullet$ as a sequence of $n+1$ consecutive spires 
\[
\EE_a \to \EE_{a+1} \to \dots \to \EE_{a+n},
\]
for some $a \in \Z$ where $\EE_j = \{E_j, E_{j+1}, \dots, E_{j+n-1} \}$.
Note that since $n$ is the period of $E_\bullet$, $\EE_{a+n}$ differs from $\EE_a$
by a twist by $\omega_X$.

We now may ask what are the helices formed by numerically exceptional collections
of Lemma \ref{numExceptional}.
The proof of the following Lemma is straightforward from definitions.

\begin{lemma}\label{helix}
Numerically exceptional helices on $S$ formed by line bundles
belong to one of the two families:
\[\bal
I_c & \to IV_c \to I_{-c} \to IV_{-c} \to I_c, \; c \in \Z \\
II_c & \to III_c \to II_{-c} \to III_{-c} \to II_c, \; c \in \Z.\\
\eal\]
\end{lemma}

\subsection{Acyclic line bundles and exceptional collections}

We will now  investigate which of the numerically exceptional collections
of Lemma \ref{numExceptional} can be lifted to exceptional collections.
Here by a lift we mean a lift with respect to the morphism
\[
\Z^2 \oplus \widehat{G} = Pic(S) \to Pic(S)/tors = \Z^2,
\]
that is a choice of a character $\chi \in \widehat{G}$.
We will need a detailed study of the characters that may appear
in the cohomology groups of sheaves on $T$.

For a $G$-linearized line bundle on $T$ we define the acyclic set of $L$ as
\[
\AA(L) := \{ \chi \in Hom(G, \C^*) \; : \; \chi \notin [H^*(T,L)] \}
\]
By definition $L(\chi)$ is acyclic if and only if $-\chi \in \AA(L)$.
Since by Proposition \ref{lineBundles}(1), any line bundle on $S$ is isomorphic to some $\KK(i,j)(\chi)$,
we see from the next lemma that there are $39$ isomorphism classes of 
acyclic line bundles on $S$.

\begin{lemma} \label{acyclicSets} The only nonempty acyclic sets of line bundles $\KK(i,j)$
on $S$ are:
\[\bal
\AA(\KK(1,-2)) &\= \{ [0,0] \} \\
\AA(\KK(1,-1)) &\= \{  [0, 3], [2, 0], [3, 2]\} \\
\AA(\KK(1,0)) &\= \{ [0, 0], [0, 1], [0, 2], [1, 4], [2, 3], [3, 0], [4, 0] \} \\
\AA(\KK(1,1)) &\= \{ [0, 0], [1, 2], [2, 1], [2, 2], [3, 3], [3, 4], [4, 3] \} \\
\AA(\KK(1,2)) &\= \{ [0, 0], [0, 3], [0, 4], [1, 0], [2, 0], [3, 2], [4, 1] \} \\
\AA(\KK(1,3)) &\= \{ [0, 2], [2, 3], [3, 0]\} \\
\AA(\KK(1,4)) &\= \{ [0, 0] \} \\
\AA(\KK(-1,1)) &\= \{ [0,0] \} \\
\AA(\KK(0,1)) &\= \{ [0, 0], [3, 3], [3, 4], [4, 3] \} \\
\AA(\KK(2,1)) &\= \{ [0, 0], [1, 2], [2, 1], [2, 2] \} \\
\AA(\KK(3,1)) &\= \{ [0, 0] \}. \\
\eal\]
\end{lemma}
\begin{proof}
Since by Proposition \ref{lineBundles}(4) any bundle $\KK(i,j)(\chi)$ with $i \ne 1$ and $j \ne 1$ 
is not acyclic we restrict to the cases $i = 1$ or $j = 1$.
We note in addition that our claim is consistent with the Serre duality: $\AA(\KK(i,j))$ is in duality with $\AA(\KK(2-i,2-j))$;
therefore we only need to consider the cases $\KK(1,j), \KK(i,1)$, $i, j \ge 1$.

For $i, j \ge 3$ we have an implication
\[
\AA(\KK(i,j)) = \emptyset \implies \AA(\KK(i+1,j)) = \emptyset, \; \AA(\KK(i,j+1)) = \emptyset,
\]
therefore it is sufficient to prove that
\be{OOacyclic}\bal
\AA(\KK(1,5)) & \= \emptyset \\
\AA(\KK(4,1)) & \= \emptyset \\
\eal\ee
and to compute $\AA(L)$ for line bundles
\[
\KK(1,1), \KK(1,2), \KK(1,3), \KK(1,4), \KK(2,1), \KK(3,1).
\]
This is done by looking at the terms of the products of the polynomials in (\ref{cohPolyn}).
\end{proof}

\begin{lemma} \label{excColl}
Let $L_1$,$L_2$,$L_3$ be line bundles on $S$.
A sequence
\[
\OO, L_1(\chi_1), L_2(\chi_2), L_3(\chi_3)
\]
forms an exceptional collection if and only if the following conditions hold:
\[\bal
\chi_1 &\in \AA(L_1^*) \\
\chi_2 &\in \AA(L_2^*) \\
\chi_3 &\in \AA(L_3^*) \\
\chi_2 - \chi_1 &\in \AA(L_1 \otimes L_2^*) \\
\chi_3 - \chi_1 &\in \AA(L_1 \otimes L_3^*) \\
\chi_3 - \chi_2 &\in \AA(L_2 \otimes L_3^*). \\
\eal\]


\end{lemma}
\begin{proof}
The statement is a reformulation of the definition of exceptional collection. 
\end{proof}

\begin{theorem}\label{collections} The following list contains all exceptional
collections of length $4$ consisting of line bundles on $S$
(up to a common twist by a line bundle):
\be{coll}\bal
(I_1) & \; \OO, \; \KK( -1 , 0 ) , \; \KK( 0 , -1 ) , \; \KK( -1 , -1 )  \\
(IV_1) & \;\OO, \; \KK( 1 , -1 ) ,  \; \KK( 0 , -1 ) , \; \KK( -1 , -2 )  \\
(I_{-1}) & \; \OO, \; \KK( -1 , 0 ) ,  \; \KK( -2 , -1 ) , \; \KK( -3 , -1 )  \\
(IV_{-1}) & \; \OO, \; \KK( -1 , -1 ) , \; \KK( -2 , -1 ) , \; \KK( -1 , -2 ) \\
(II_0=IV_0) & \; \OO, \; \KK( 0 , -1 ) ,  \; \KK( -1 , -1 ), \; \KK( -1 , -2 ) \\
(I_0) & \; \OO, \; \KK( -1 , 0 ) ,  \; \KK( -1 , -1 ) , \; \KK( -2 , -1 ).  \\
\eal\ee
These six collections are spires of the two helices 
\be{H12}\bal
(\HH_1) \; & I_1 \to IV_1 \to I_{-1} \to IV_{-1} \to I_1 \\
(\HH_2) \; & I_0 \to II_0 \to I_0. \\
\eal\ee
\end{theorem}
\begin{proof}
Because of Remark \ref{helix} we only need to consider numerically exceptional collections
of types $I_c, c \ge 0$, $II_c, c > 0$ and all helices formed by them.
Let us start by listing all numerically exceptional collections 
\[
\OO, L_1, L_2, L_3 = L_1 \otimes L_2
\]
of line bundles of the types as above satisfying the properties:
\[
\AA(L_1^*) \ne \emptyset; \; \AA(L_2^*) \ne \emptyset; \; \AA(L_3^*) \ne \emptyset
\]
\[
\AA(L_1 \otimes L_2^*) \ne \emptyset.
\]
By Lemma \ref{acyclicSets} these properties are necessary for $\OO, L_1, L_2, L_3$ to form an exceptional collection.
With the help of Lemmas \ref{numExceptional} and \ref{acyclicSets} we get the following list:
\[
I_0, I_1, II_1, II_2. 
\]

Finally we check whether there are characters $\chi_1, \chi_2, \chi_3$ for each
of these types of collections that will satisfy the conditions of Lemma \ref{excColl}.

{\bf Type $I_0$}: $\OO, \KK(-1,0)(\chi_1), \KK(-1,-1)(\chi_2), \KK(-2,-1)(\chi_3)$ with conditions
\[\bal
\chi_1, \chi_3 - \chi_2 & \in \AA(\KK(1,0)) = \{ [0, 0], [0, 1], [0, 2], [1, 4], [2, 3], [3, 0], [4, 0] \}\\  
\chi_2, \chi_3 - \chi_1 & \in \AA(\KK(1,1)) = \{ [0, 0], [1, 2], [2, 1], [2, 2], [3, 3], [3, 4], [4, 3] \} \\
\chi_3 & \in \AA(\KK(2,1)) = \{ [0, 0], [1, 2], [2, 1], [2, 2] \} \\
\chi_2 - \chi_1 & \in \AA(\KK(0,1)) =  \{ [0, 0], [3, 3], [3, 4], [4, 3] \} \\
\eal\]

For each choice of $\chi_3$ we find possible $\chi_1, \chi_2$ from conditions
\be{cond1}\bal
\chi_1 & \in \AA(\KK(1,0)) \cap \chi_3 - \AA(\KK(1,1)) \\
\chi_2 & \in \AA(\KK(1,1)) \cap \chi_3 - \AA(\KK(1,0)) \\
\eal\ee
and look for those $\chi_1, \chi_2$ that satisfy
\be{cond2}
\chi_2 - \chi_1 \in \AA(\KK(0,1)). 
\ee

1. $\chi_3 = [0,0]$.
Using (\ref{cond1}) we find the only set of characters
\[
\chi_1 = \chi_2 = [0,0]
\]
and it obviously satisfies the condition (\ref{cond2}) as well.
Thus we obtain the collection
\[
(I_0) \; \OO, \; \KK(-1,0),  \; \KK(-1,-1), \; \KK(-2,-1) \\
\]
and the one in the same helix
\[
(II_0=IV_0) \; \OO, \; \KK( 0 , -1 ) ,  \; \KK( -1 , -1 ), \; \KK( -1 , -2 ). \\
\]

2. $\chi_3 = [1,2]$.
(\ref{cond1}) reads as:
\[\bal
\chi_1 & \in \{ [0,0], [4,0] \} \\
\chi_2 & \in \{ [1,2], [2,2] \} \\
\eal\]
and none of these pairs satisfies (\ref{cond2}).

3. $\chi_3 = [2,1]$.
(\ref{cond1}) reads as:
\[\bal
\chi_1 & \in \{ [0,0], [1,4] \} \\
\chi_2 & \in \{ [1,2], [2,1] \} \\
\eal\]
and none of these pairs satisfies (\ref{cond2}).

4. $\chi_3 = [2,2]$
(\ref{cond1}) reads as:
\[\bal
\chi_1 & \in \{ [0,1], [0,0] \} \\
\chi_2 & \in \{ [2,1], [2,2] \} \\
\eal\]
and none of these pairs satisfies (\ref{cond2}).

\medskip

{\bf Type $I_1$}: $\OO, \KK(-1,0)(\chi_1), \KK(0,-1)(\chi_2), \KK(-1,-1)(\chi_3)$ with conditions
\[\bal
\chi_1, \chi_3 - \chi_2 & \in \AA(\KK(1,0)) = \{ [0, 0], [0, 1], [0, 2], [1, 4], [2, 3], [3, 0], [4, 0] \}\\  
\chi_2, \chi_3 - \chi_1 & \in \AA(\KK(0,1)) = \{ [0, 0], [3, 3], [3, 4], [4, 3] \} \\
\chi_3 & \in \AA(\KK(1,1)) =  \{ [0, 0], [1, 2], [2, 1], [2, 2], [3, 3], [3, 4], [4, 3] \}\\
\chi_2 - \chi_1 & \in \AA(\KK(-1,1)) = \{ [0,0] \}\\
\eal\]

From the conditions on $\chi_1, \chi_2$ we find that $\chi_1 = \chi_2 = [0,0]$. Then
\[
\chi_3 \in \AA(\KK(1,1)) \cap \AA(\KK(1,0)) \cap \AA(\KK(0,1)) = \{ [0,0] \}.
\]

This way we get exceptional collection
\[\bal
(I_1) & \; \OO, \; \KK(-1,0), \; \KK(0,-1) , \; \KK(-1,-1)  \\
\eal\]
and three others lying in the same helix
\[\bal
(IV_1) & \;\OO, \; \KK( 1 , -1 ) ,  \; \KK( 0 , -1 ) , \; \KK( -1 , -2 )  \\
(I_{-1}) & \; \OO, \; \KK( -1 , 0 ) ,  \; \KK( -2 , -1 ) , \; \KK( -3 , -1 )  \\
(IV_{-1}) & \; \OO, \; \KK( -1 , -1 ) , \; \KK( -2 , -1 ) , \; \KK( -1 , -2 ). \\
\eal\]

\medskip

{\bf Type $II_1$}: $\OO, \KK(0,-1)(\chi_1), \KK(-1,0)(\chi_2), \KK(-1,-1)(\chi_3)$ with conditions
\[\bal
\chi_1, \chi_3 - \chi_2 & \in \AA(\KK(0,1)) = \{  [0, 0], [3, 3], [3, 4], [4, 3] \}\\  
\chi_2, \chi_3 - \chi_1 & \in \AA(\KK(1,0)) = \{ [0, 0], [0, 1], [0, 2], [1, 4], [2, 3], [3, 0], [4, 0] \} \\
\chi_3 & \in \AA(\KK(1,1)) =  \{ [0, 0], [1, 2], [2, 1], [2, 2], [3, 3], [3, 4], [4, 3] \}\\
\chi_2 - \chi_1 & \in \AA(\KK(1,-1)) = \{ [0, 3], [2, 0], [3, 2] \}\\
\eal\]

There exist no $\chi_1$, $\chi_2$ satisfying the respective conditions.

\medskip

{\bf Type $II_2$}: $\OO, \OO(0,-1)(\chi_1), \OO(-1,1)(\chi_2), \OO(-1,0)(\chi_3)$ with conditions
\[\bal
\chi_1, \chi_3 - \chi_2 & \in \AA(\KK(0,1)) = \{  [0, 0], [3, 3], [3, 4], [4, 3] \} \\   
\chi_2, \chi_3 - \chi_1 & \in \AA(\KK(1,-1)) = \{ [0, 3], [2, 0], [3, 2] \}\\
\chi_3 & \in \AA(\KK(1,0)) =  \{ [0, 0], [0, 1], [0, 2], [1, 4], [2, 3], [3, 0], [4, 0] \}\\
\chi_2 - \chi_1 & \in \AA(\KK(1,-2)) = \{ [0, 0] \}\\
\eal\]

There exist no $\chi_1$, $\chi_2$ satisfying the respective conditions.

\end{proof}

\begin{remark}\label{canLattice} All six exceptional collections in (\ref{coll})
span the same torsion-free subgroup in $Pic(S)$ with generators
\[\bal
\KK(1,0) &\= \OO(1,0)[3,3], \\
\KK(0,1) &\= \OO(0,1)[3,2].
\eal\]
We do not have a conceptual proof for this statement.
\end{remark}

For a helix $E_\bullet$ of period $n$ we introduce a matrix $\MM(E_\bullet)$ with entries 
consisting of the $Ext$-groups in the spires of $E_\bullet$:
\[
M_{i,j} = \sum_l \dim Ext^l(E_i, E_{i+j}) \cdot q^l ;\ 0 \le i,j \le n-1. 
\]

\begin{proposition}\label{algebras}
For the helices (\ref{H12}) we have:
\[\bal
\MM(\HH_1) & \= 
\begin{pmatrix} 
1 &  3q^2 + q & 3q^2 + q & 4q^2 \\
1 &  3q^2 + 3q & 3q^2 + q & 6q^2 \\
1 &  3q^2 + q & 6q^2 & 8q^2 \\
1 & 4q^2 & 6q^2 & 6q^2 \\
\end{pmatrix} \\
\MM(\HH_2) & \= 
\begin{pmatrix} 
1 &  3q^2 + q & 4q^2 & 6q^2 \\
1 &  3q^2 + q & 4q^2 & 6q^2 \\
1 &  3q^2 + q & 4q^2 & 6q^2 \\
1 &  3q^2 + q & 4q^2 & 6q^2 \\
\end{pmatrix} \\
\eal\]
In particular we see that all our collections have endomorphism $dg$-algebras with non-vanishing
first and second cohomology groups.
\end{proposition}
\begin{proof}
The entries are found in the table given in Lemma \ref{cohKK}.
\end{proof}

\begin{proposition}
The $A_{\infty}$-algebra of the exceptional collection 
\[
(I_{-1}) \; \OO, \; \KK( -1 , 0 ) ,  \; \KK( -2 , -1 ), \; \KK( -3 , -1 ) 
\]
is formal and moreover the usual product $m_2$ is trivial.
\end{proposition}
\begin{proof}
The $Ext$-groups of the collection $E_\bullet = I_{-1}$ are 
all found in $\MM(\HH_1)$ from the previous Proposition.
In fact we have:
\[\Bigl(\sum_l \dim Ext^l(E_i, E_j) \cdot q^l\Bigr)_{i,j} = 
\begin{pmatrix} 
1 &  3q^2 + q & 6q^2 & 8q^2 \\
0 & 1 &  4q^2 & 6q^2 \\
0 & 0 & 1 &  3q^2 + q\\
0 & 0 & 0 & 1\\
\end{pmatrix}. \\
\]

In order to prove formality we check that the higher $A_\infty$-operations $m_k, k \ge 3$
of the collection $E_\bullet$ vanish.

Using a standard argument (see [Sei08] Lemma 2.1 or [Lef02] Th 3.2.1.1),
we may assume that $m_l (\dots, id_{E_i} , \dots) = 0$ for all objects $E_i$
and all $l > 2$.
Now the third $A_\infty$-operation $m_3$
vanishes for grading reasons and the products $m_k, k \geq 4$ also vanish
since our graded quiver has only $4$ vertices.

The product of the two non-trivial elements of degree 1 vanishes since
these elements are not composable. All other products are trivial for
grading reasons.

\end{proof}

\medskip

Let $(E_\bullet)$ be one of the collections in (\ref{coll}).
By \cite{BK}, Theorem 2.10 
the subcategory $\left<E_0,E_1,E_2,E_3\right>$ generated by the collection is admissible and 
has a right orthogonal $\AA$, i.e. there is a semiorthogonal decomposition
$D^b_{coh}(S) = \left<E_0,E_1,E_2,E_3,\AA\right>$.

\begin{proposition}
Right orthogonals to two spires of a helix are equivalent categories.
\end{proposition}
\begin{proof}
By transitivity it is enough to prove the statement for two consecutive spires.
Denote $E_4 = E_0 \otimes \omega_S^{-1}$.
Let $\AA$ be the right orthogonal to $\left< E_0, E_1, E_2, E_3 \right>$,
and $\AA'$ be the right orthogonal to $\left< E_1, E_2, E_3, E_4 \right>$.
We want to show that categories $\AA$ and $\AA'$ are equivalent.
Denote by $\CC$ the right orthogonal to $\left< E_1, E_2, E_3 \right>$.
Second decomposition $D^b(S) = \left< E_1, E_2, E_3, E_4, \AA' \right>$
implies $\CC = \left< E_4, \AA' \right>$.
First decomposition $D^b(S) = \left< E_0, E_1, E_2, E_3, \AA \right>$ is equivalent to
$D^b(S) = \left< E_1, E_2, E_3, \AA, E_4 \right>$
by Serre duality, so $\CC = \left< \AA, E_4 \right>$.
Hence both $\AA$ and $\AA'$ are subcategories in $\CC$ orthogonal to $E_4$:
$\AA$ is the left orthogonal and $\AA'$ is the right orthogonal.
So (left/right) mutations in $E_4$ establish the equivalence between $\AA$ and $\AA'$.
\end{proof}

We denote two equivalence classes of subcategories obtained by taking right orthogonals to $\HH_1$ and $\HH_2$ by $\AA_1$ and $\AA_2$
respectively.
We note that a choice of a spire gives rise to a fully faithful embedding $\AA_i \to D^b(S)$.

\begin{proposition} We have
\[\bal
K_0(\AA_i) &\= (\Z/5)^2 \\
HH_*(\AA_i) &\= 0 \\
HH^0(\AA_i) &\= \C. 
\eal\]
In particular we see that $\AA_i$'s are non-trivial.
\end{proposition}
\begin{proof}
We have
\[
\Z^4 \oplus (\Z/5)^2 = K_0(S) = K_0(D^b(S)) = K_0(\left<E_0, E_1, E_2, E_4\right>) \oplus K_0(\AA_i) = \Z^4 \oplus K_0(\AA_i),
\]
thus
\[
K_0(\AA_i) \cong (\Z/5)^2.
\]

For the homology we use the additivity theorem \cite{Ke} (see also\cite{Kuz09}, Corollary 7.5):
\[
\C^4 = H^*(S) = HH_*(D^b(S)) = HH_*(\left<E_0, E_1, E_2, E_4\right>) \oplus HH_*(\AA_i) = \C^4 \oplus HH_*(\AA_i).
\]

The statement about Hochschild cohomology is proved by the following approach of Kuznetsov \cite{Kuz12}.
Define
\[ e(F,F') = \min \{ p \; | \; Ext^p(F,F') \neq 0 \} \]
For any increasing sequence $ a_0 < a_1< \dots < a_k = a_0 + n $
($n$ is the period of the helix $E_{\bullet}$, in
our case $n=4$)
define
\[ \delta_{a_\bullet}(E_\bullet) = e(E_{a_0},E_{a_1}) + \dots +
e(E_{a_{k-1}},E_{a_k}) + 1-k. \]

Finally the anticanonical height of the exceptional collection is defined as
\[ h(E_\bullet) = \min_{a_\bullet} \delta_{a_\bullet} (E_\bullet) \]

We now use the following result:
\begin{proposition}\cite{Kuz12} 
 Let $\AA$ be right orthogonal to exceptional
collection $E_\bullet$. For $k \leq h(E_\bullet) + (\dim S - 2)$ the natural map
$HH^k(S) \to HH^k(\AA)$ is isomorphism.
\end{proposition}

For our helices we have
\[\bal
h(\HH_1) &\= 2 \\
h(\HH_2) &\= 1 \\
\eal\]
and hence we see that $HH^0(\AA_i) = HH^0(S) = \C$.

\end{proof}

\providecommand{\arxiv}[1]{\href{http://arxiv.org/abs/#1}{\tt arXiv:#1}}

\end{document}